\documentclass[11pt]{amsart}
\usepackage{
amssymb
,amsthm
,amsmath
,amscd
,mathtools,
mathdots,
url
}
\usepackage[all]{xy}
\usepackage[top=35truemm,bottom=35truemm,left=30truemm,right=30truemm]{geometry}

\newcommand{\Z}{\mathbb{Z}}
\newcommand{\R}{\mathbb{R}}
\newcommand{\C}{\mathbb{C}}

\newcommand{\Sc}{\mathcal{S}}
\newcommand{\EE}{\mathcal{E}}

\newcommand{\DD}{\mathcal{D}}

\newcommand{\GL}{\mathrm{GL}}
\newcommand{\SL}{\mathrm{SL}}
\newcommand{\SO}{\mathrm{SO}}
\newcommand{\Sp}{\mathrm{Sp}}

\newcommand{\Irr}{\mathrm{Irr}}
\newcommand{\unit}{\mathrm{unit}}
\newcommand{\gp}{\mathrm{gp}}
\newcommand{\temp}{\mathrm{temp}}

\newcommand{\Jac}{\mathrm{Jac}}
\newcommand{\Cusp}{\mathrm{Cusp}}
\newcommand{\esupp}{\mathrm{ex.supp}}

\newcommand{\iif}{&\quad&\text{if }}

\newcommand{\resp}{resp.~}
\renewcommand{\1}{\mathbf{1}}
\newcommand{\ep}{\varepsilon}

\newcommand{\half}[1]{\frac{#1}{2}}

\newtheorem{thm}{Theorem}[section]

\newtheorem{prop}[thm]{Proposition}

\newtheorem{defi}[thm]{Definition}

\newtheorem{alg}[thm]{Algorithm}
\newtheorem{problem}[thm]{Problem}

\title{The set of local $A$-packets 
containing a given representation}
\author{Hiraku Atobe}
\date{}
\address{
Department of Mathematics, Hokkaido University,
Kita 10, Nishi 8, Kita-Ku, Sapporo, Hokkaido, 060-0810, Japan 
}
\email{
atobe@math.sci.hokudai.ac.jp
}

\allowdisplaybreaks
\begin{document}
\maketitle

\begin{abstract}
In this paper, 
we give an algorithm to 
determine all local $A$-packets 
containing a given irreducible representation of a $p$-adic classical group. 
Especially, we can determine 
whether a given irreducible representation is of Arthur type or not. 
\end{abstract}

\section{Introduction}\label{intro}
In a magnificent work \cite{Ar}, 
Arthur gave a classification of the discrete spectrum of 
square integral automorphic forms on
a quasi-split special orthogonal group $\SO_n$ or a symplectic group $\Sp_{2n}$ 
over a number field.
This classification says that the discrete spectrum of 
the space of square integrable automorphic forms 
is divided into disjoint subsets called \emph{global $A$-packets}. 
An element in a global $A$-packet 
is defined as the tensor product of elements of \emph{local $A$-packets}.
In other words, the local $A$-packets classify 
the local factors of square integrable automorphic representations.
The purpose of this paper is 
to give an algorithm to determine all local $A$-packets 
which contain a given irreducible representation. 
This result might be regarded as an ultimate form of 
the globalization of irreducible representations of $p$-adic classical groups.
\vskip 10pt

Let us describe our results.
Fix a non-archimedean local field $F$ of characteristic zero.
Let $G_n$ be a split special odd orthogonal group $\SO_{2n+1}(F)$ 
or a symplectic group $\Sp_{2n}(F)$ of rank $n$ over $F$. 
A (local) $A$-parameter for $G_n$ is a homomorphism 
\[
\psi \colon W_F \times \SL_2(\C) \times \SL_2(\C) \rightarrow \widehat{G}_n
\]
such that $\psi(W_F)$ is bounded, 
where $W_F$ is the Weil group of $F$
and $\widehat{G}_n$ is the complex dual group of $G_n$. 
Associated to an $A$-parameter $\psi$ for $G_n$, 
Arthur \cite[Theorem 1.5.1]{Ar} defined the (local) $A$-packet $\Pi_\psi$, 
which is a multi-set over $\Irr_{\unit}(G_n)$. 
Here, $\Irr_{(\unit)}(G_n)$ denotes 
the set of equivalence classes of irreducible (unitary) representations of $G_n$. 
In fact, M{\oe}glin \cite{Moe11c} showed that the $A$-packet $\Pi_\psi$ is multiplicity-free, 
i.e., $\Pi_\psi$ is a subset of $\Irr_{\unit}(G_n)$. 
We say that an irreducible representation $\pi$ of $G_n$ is \emph{of Arthur type}
if $\pi$ belongs to $\Pi_\psi$ for some $A$-parameter $\psi$ for $G_n$. 
\vskip 10pt

Unlike $L$-packets, $A$-packets do not give a classification of $\Irr_{\unit}(G_n)$. 
Indeed, there is an irreducible unitary representation $\pi$ which is not of Arthur type. 
Moreover, even if $\psi_1 \not\cong \psi_2$, 
the intersection $\Pi_{\psi_1} \cap \Pi_{\psi_2}$ is not necessarily empty.  
Therefore, the following problem occurs. 

\begin{problem}\label{pro1}
Let $\pi$ be an irreducible representation of $G_n$. 
\begin{enumerate}
\item
Determine whether $\pi$ is of Arthur type or not. 
\item
If $\pi$ is of Arthur type, 
list all $A$-parameters $\psi$ such that $\pi \in \Pi_\psi$. 
\end{enumerate}
\end{problem}

By a result of M{\oe}glin ({\cite[Theorem 6]{Moe06a}, \cite[Proposition 8.11]{X2}}), 
this problem is reduced to the good parity case. 
Here, we say that an irreducible representation $\pi$ of $G_n$ is \emph{of good parity} 
if $\pi$ can be embedded into a parabolically induced representation 
$\rho_1 \times \dots \times \rho_r \rtimes \sigma$ such that
\begin{itemize}
\item
$\rho_i$ is an irreducible cuspidal representation of $\GL_{d_i}(F)$
for $i=1,\dots,r$; 
\item
$\sigma$ is an irreducible cuspidal representation of 
a classical group $G_{n_0}$ of the same type as $G_n$; 
\item
for each $i=1,\dots,r$, 
there exists an \emph{integer} $m_i$ such that 
the parabolically induced representation $\rho_i|\cdot|^{m_i} \rtimes \sigma$ is reducible.
\end{itemize}
\vskip 10pt

As a refinement of M{\oe}glin's explicit construction of $A$-packets, 
in \cite{At}, 
we introduced a notion of \emph{extended multi-segments} $\EE$ for $G_n$. 
See Definition \ref{segments}. 
An extended multi-segment $\EE$ for $G_n$ explicitly gives 
\begin{itemize}
\item
a representation $\pi(\EE)$ of $G_n$, 
which is zero or an irreducible representation of good parity; 
and
\item
an $A$-parameter $\psi_\EE$ for $G_n$ of good parity.
\end{itemize}
Here, we say that an $A$-parameter $\psi$ for $G_n$ is \emph{of good parity} 
if all irreducible components of $\psi$ are self-dual of the same type as $\psi$. 
Moreover, for an $A$-parameter $\psi$ for $G_n$ of good parity, 
(after fixing an auxiliary datum)
the $A$-packet $\Pi_\psi$ is given as
\[
\Pi_\psi = \{\pi(\EE) \;|\; \psi_\EE \cong \psi\} \setminus \{0\}. 
\]
Therefore, Problem \ref{pro1} for irreducible representations of good parity
can be reformulated as follows: 

\begin{problem}\label{pro2}
Let $\pi$ be an irreducible representation of $G_n$ of good parity. 
\begin{enumerate}
\item
Determine whether 
there is an extended multi-segment $\EE$ such that $\pi(\EE) \cong \pi$. 
\item
In the affirmative case, 
determine all extended multi-segments $\EE$ such that $\pi(\EE) \cong \pi$. 
\end{enumerate}

\end{problem}

To solve this problem, 
we need two key notions.
The one is \emph{derivatives} and the other is \emph{extended cuspidal supports}. 
Fix an irreducible unitary cuspidal representation $\rho$ of $\GL_d(F)$ 
and a real number $x$. 
For an irreducible representation $\pi$ of $G_n$, 
define the \emph{$\rho|\cdot|^x$-derivative} $D_{\rho|\cdot|^x}(\pi)$
as a semisimple representation satisfying that 
\[
[\Jac_P(\pi)] = \rho|\cdot|^x \otimes D_{\rho|\cdot|^x}(\pi) + (\text{others}), 
\]
where $[\Jac_P(\pi)]$ is the semisimplification of 
the Jacquet module of $\pi$
along a standard parabolic subgroup $P$ of $G_n$ 
with Levi $\GL_d(F) \times G_{n-d}$. 
In \cite[Section 5]{At}, a composition of derivatives of $\pi(\EE)$
was described in terms of extended multi-segments. 
On the other hand, 
the \emph{extended cuspidal support} $\esupp(\pi)$ of 
an irreducible representation $\pi$ of $G_n$
is a refinement of the usual cuspidal support of $\pi$
(see Definition \ref{def.esupp}).
Unlike the usual cuspidal support, 
members of an $A$-packet $\Pi_\psi$ share an extended cuspidal support, 
which is determined by the diagonal restriction $\psi_d$ of $\psi$
(see Proposition \ref{ex.supp}). 
\vskip 10pt

Now we can roughly state our first main result, 
which solves Problem \ref{pro2} (1). 

\begin{alg}[Algorithm \ref{main.alg}]
Let $\pi$ be an irreducible representation of $G_n$ of good parity. 

\begin{description}
\item[Step $1^\pm$]
Suppose that there exist 
an irreducible unitary cuspidal representation $\rho$ of $\GL_d(F)$ 
and $x \geq 1$ or $x < 0$
such that $D_{\rho|\cdot|^{x}}(\pi) \not= 0$. 
In this case, one can construct 
an irreducible representation
$\pi^\pm$ of $G_{n^\pm}$ of good parity with $n^\pm < n$ 
by Definition \ref{d.d}.
Then $\pi$ is of Arthur type 
if and only if 
there exists an extended multi-segment $\EE^\pm$ for $G_{n^\pm}$
satisfying certain conditions such that $\pi(\EE^\pm) \cong \pi^\pm$. 
In this case, we can explicitly define $\EE$ from $\EE^\pm$ such that $\pi(\EE) \cong \pi$.

\item[Step $2$]
Otherwise, 
the extended cuspidal support $\esupp(\pi)$ of $\pi$ 
gives at most one $A$-parameter $\psi$ for $G_n$
such that 
$\pi$ is of Arthur type if and only if $\pi \in \Pi_\psi$. 
\end{description}
\end{alg}

To apply this algorithm, 
we need to solve Problem \ref{pro2} (2).
The following is the second main result. 
\begin{thm}[Theorem \ref{thm:s-eq}]
Let $\EE_1$ and $\EE_2$ be two extended multi-segments for $G_n$. 
Suppose that $\pi(\EE_1) \not= 0$. 
Then $\pi(\EE_1) \cong \pi(\EE_2)$ if and only if 
$\EE_2$ can be obtained from $\EE_1$
by  a finite chain of three operations (C), (UI), (P) defined in Definition \ref{CUIP}
and their inverses. 
\end{thm}
\vskip 10pt

This paper is organized as follows.
In Section \ref{s.pre}, we review 
$A$-parameters, $A$-packets, extended cuspidal supports, 
Langlands classification, and derivatives. 
After recalling results in the previous paper \cite{At}, 
we state main results and give some examples in Section \ref{s.main}. 
Finally, in Section \ref{s.proof}, 
we prove the main results.
\vskip 10pt

\noindent
{\bf Acknowledgement.}
We would like to thank Alberto M{\'i}nguez for helpful discussions and communication.
The author was supported by JSPS KAKENHI Grant Number 19K14494. 
\vskip 10pt

\noindent
\textbf{Notation.}
Let $F$ be a non-archimedean local field of characteristic zero.
The normalized absolute value is denoted by $|\cdot|$, 
which is also regarded as a character of $\GL_d(F)$ via composing with the determinant map. 
\par

Let $G_n$ be a split special odd orthogonal group $\SO_{2n+1}(F)$ or a symplectic group $\Sp_{2n}(F)$ 
of rank $n$ over $F$. 
The set of equivalence classes of irreducible smooth representations of a group $G$ is denoted by $\Irr(G)$. 
Let $\Irr_\unit(G_n)$ (\resp $\Irr_\temp(G_n)$) be the subset of $\Irr(G_n)$ consisting of
equivalence classes of irreducible unitary (\resp tempered) representations of $G_n$.
\par

The Weil group of $F$ is denoted by $W_F$. 
The group $\SL_2(\C)$ has a unique irreducible algebraic representation of dimension $a$, 
which is denoted by $S_a$.
We denote by $\widehat{G}_n$ the complex dual group of $G_n$. 
Namely, $\widehat{G}_n = \Sp_{2n}(\C)$ if $G_n = \SO_{2n+1}(F)$, and 
$\widehat{G}_n = \SO_{2n+1}(\C)$ if $G_n = \Sp_{2n}(F)$.  
\par

The set of equivalence classes of irreducible cuspidal representations of $\GL_d(F)$
is denoted by $\Cusp(\GL_d(F))$.
By the local Langlands correspondence for $\GL_d(F)$, 
we identify $\rho \in \Cusp(\GL_d(F))$ with 
an irreducible $d$-dimensional representation of $W_F$.
The subset of $\Cusp(\GL_d(F))$ consisting of unitary (\resp self-dual) elements 
is denoted by $\Cusp_\unit(\GL_d(F))$ (\resp $\Cusp^\bot(\GL_d(F))$).
\par

We will often extend the set theoretical language to multi-sets. 
Namely, we write a multi-set as $\{x, \dots, x, y, \dots, y, \ldots \}$.
When we use a multi-set, we will mention it.

\section{Preliminary}\label{s.pre}
In this section, we review several results on local $A$-packets.

\subsection{$A$-parameters}
Recall that an \emph{$A$-parameter for $G_n$} is 
the $\widehat{G}_n$-conjugacy class of an admissible homomorphism
\[
\psi \colon W_F \times \SL_2(\C) \times \SL_2(\C) \rightarrow \widehat{G}_n
\]
such that the image of $W_F$ is bounded.
By composing with the standard representation of $\widehat{G}_n$, 
we can regard $\psi$ as a representation of $W_F \times \SL_2(\C) \times \SL_2(\C)$. 
We write
\[
\psi = \bigoplus_\rho\left(\bigoplus_{i \in I_\rho} \rho \boxtimes S_{a_i} \boxtimes S_{b_i}\right), 
\]
where $\rho$ runs over $\sqcup_{d \geq 1}\Cusp_\unit(\GL_d(F))$.
\par

For $\psi$ as above, 
we say that 
$\psi$ is \emph{of good parity} 
if $\rho \boxtimes S_{a_i} \boxtimes S_{b_i}$ is self-dual of the same type as $\psi$ 
for any $\rho$ and $i \in I_\rho$, 
i.e., 
\begin{itemize}
\item
$\rho \in \Cusp^\bot(\GL_d(F))$ is orthogonal and $a_i+b_i \equiv 0 \bmod 2$ 
if $G_n = \Sp_{2n}(F)$
(\resp $a_i+b_i \equiv 1 \bmod 2$ if $G_n = \SO_{2n+1}(F)$); or 
\item
$\rho \in \Cusp^\bot(\GL_d(F))$ is symplectic and $a_i+b_i \equiv 1 \bmod 2$ 
if $G_n = \Sp_{2n}(F)$
(\resp $a_i+b_i \equiv 0 \bmod 2$ if $G_n = \SO_{2n+1}(F)$).
\end{itemize}
Let $\Psi(G_n)$ be the set of $A$-parameters. 
The subset of $\Psi(G_n)$ consisting of $A$-parameters of good parity is denoted by $\Psi_\gp(G_n)$. 
Also, we denote by $\Phi_\temp(G_n)$ the subset of $\Psi(G_n)$ consisting of
\emph{tempered} $A$-parameters, i.e., 
$A$-parameters $\phi$ which are trivial on the second $\SL_2(\C)$. 
Finally, we set $\Phi_\gp(G_n) = \Psi_\gp(G_n) \cap \Phi_\temp(G_n)$.
\par

\subsection{$A$-packets}\label{Apacket}
To an $A$-parameter $\psi \in \Psi(G_n)$, 
Arthur \cite[Theorem 1.5.1 (a)]{Ar} associated an \emph{$A$-packet} $\Pi_\psi$, 
which is a finite multi-set over $\Irr_\unit(G_n)$.
In fact, M{\oe}glin \cite{Moe11c} showed that $\Pi_\psi$ is multiplicity-free, 
i.e., a subset of $\Irr_\unit(G_n)$. 
We say that $\pi \in \Irr(G_n)$ is \emph{of Arthur type} if $\pi \in \Pi_\psi$ for some $\psi \in \Psi(G_n)$. 
\par

Recall that for a representation $\psi_1$ of $W_F \times \SL_2(\C) \times \SL_2(\C)$, 
we have an irreducible unitary representation $\tau_{\psi_1}$ 
of $\GL_m(F)$ with $m = \dim(\psi)$,
which is a product of unitary Speh representations. 

\begin{prop}[M{\oe}glin ({\cite[Theorem 6]{Moe06a}, \cite[Proposition 8.11]{X2}})]\label{bad}
Any $\psi \in \Psi(G_n)$ can be decomposed as 
\[
\psi = \psi_1 \oplus \psi_0 \oplus \psi_1^\vee, 
\]
where 
\begin{itemize}
\item
$\psi_0 \in \Psi_\gp(G_{n_0})$; 
\item
$\psi_1$ is a direct sum of irreducible representations of $W_F \times \SL_2(\C) \times \SL_2(\C)$
which are not self-dual of the same type as $\psi$. 
\end{itemize}
For $\pi_0 \in \Pi_{\psi_0}$, 
the parabolically induced representation $\tau_{\psi_1} \rtimes \pi_0$
is irreducible and independent of the choice of $\psi_1$. 
Moreover, 
\[
\Pi_\psi = \left\{ \tau_{\psi_1} \rtimes \pi_0 \;\middle|\; \pi_0 \in \Pi_{\psi_0} \right\}.
\]
\end{prop}
\par

By \cite[Theorem 1.5.1 (b)]{Ar}, if $\phi \in \Phi_\temp(G_n)$ is a tempered $A$-parameter, 
then $\Pi_\phi$ is a subset of $\Irr_\temp(G_n)$ and
\[
\Irr_\temp(G_n) = \bigsqcup_{\phi \in \Phi_\temp(G_n)} \Pi_\phi
\quad \text{(disjoint union)}.
\]
Moreover, after fixing a Whittaker datum for $G_n$, 
the tempered $A$-packet $\Pi_\phi$ is parametrized by $\widehat{\Sc_\phi}$, 
which is the Pontryagin duals of the component group $\Sc_\phi$. 
When $\phi = \oplus_\rho(\oplus_{i \in I_\rho} \rho \boxtimes S_{a_i}) \in \Phi_\gp(G_n)$, 
an element $\ep \in \widehat{\Sc_\phi}$ is characterized by 
$\ep(\rho \boxtimes S_{a_i}) \in \{\pm1\}$ for $\rho$ and $i \in I_\rho$ such that 
\begin{itemize}
\item
if $a_i = a_j$, 
then $\ep(\rho \boxtimes S_{a_i}) = \ep(\rho \boxtimes S_{a_j})$; 
\item
$\prod_\rho \prod_{i \in I_\rho} \ep(\rho \boxtimes S_{a_i}) = 1$.
\end{itemize}
We denote the element of $\Pi_\phi$ corresponding to $\ep \in \widehat{\Sc_\phi}$ by $\pi(\phi, \ep)$.
\par

\subsection{Extended cuspidal support}

One of key notions in this paper is as follows. 
\begin{defi}[M{\oe}glin]\label{def.esupp}
Let $\pi \in \Irr(G_n)$. 
Take $\rho_i \in \Cusp(\GL_{d_i}(F))$ for $1 \leq i \leq r$ and 
an irreducible cuspidal representation $\sigma$ of $G_{n_0}$ 
such that 
\[
\pi \hookrightarrow \rho_1 \times \dots \times \rho_r \rtimes \sigma.
\]
Write $\sigma = \pi(\phi_\sigma, \ep_\sigma)$ 
with 
\[
\phi_\sigma = \bigoplus_{j=1}^t \rho'_j \boxtimes S_{a_j} \in \Phi_\temp(G_{n_0}).
\]
Then we define an \emph{extended cuspidal support} $\esupp(\pi)$
as the multi-set 
\begin{align*}
\esupp(\pi) &= \{\rho_1,\dots,\rho_r, \rho_1^\vee, \dots, \rho_r^\vee\}
\\&\sqcup
\bigsqcup_{j=1}^t
\{\rho_j'|\cdot|^{\half{a_j-1}},\rho_j'|\cdot|^{\half{a_j-3}}, \dots, \rho_j'|\cdot|^{-\half{a_j-1}}\}
\end{align*}
over $\sqcup_{d \geq 1}\Cusp(\GL_{d}(F))$.
Here, $\rho^\vee$ denotes the contragredient of $\rho$.
\end{defi}

Let $\psi \in \Psi(G_n)$. 
Recall that $\psi$ is a homomorphism from $W_F \times \SL_2(\C) \times \SL_2(\C) \rightarrow \widehat{G}_n$.
Define the \emph{diagonal restriction} of $\psi$ by $\psi_d = \psi \circ \Delta$, 
where 
\[
\Delta \colon W_F \times \SL_2(\C) \rightarrow W_F \times \SL_2(\C) \times \SL_2(\C),\;
(w,\alpha) \mapsto (w,\alpha,\alpha). 
\]
The following is a key proposition. 

\begin{prop}[{\cite[4.1 Proposition]{Moe-comp}}]\label{ex.supp}
Let $\psi \in \Psi(G_n)$. 
Write $\psi_d = \oplus_{i=1}^r \rho_i \boxtimes S_{a_i}$. 
Then for any $\pi \in \Pi_\psi$, 
the extended cuspidal support of $\pi$ is given by 
\[
\esupp(\pi) = \bigsqcup_{i=1}^r
\{\rho_i|\cdot|^{\half{a_i-1}},\rho_i|\cdot|^{\half{a_i-3}}, \dots, \rho_i|\cdot|^{-\half{a_i-1}}\}.
\]
\end{prop}

As a consequence, 
M{\oe}glin showed that 
if $\Pi_{\psi_1} \cap \Psi_{\psi_2} \not= \emptyset$, then $\psi_{1,d} \cong \psi_{2,d}$ 
(\cite[4.2 Corollaire]{Moe-comp}).

\subsection{Langlands classification}
A \emph{segment} is a finite set 
consisting of cuspidal representations of $\GL_d(F)$
of the form
\[
[x,y]_\rho = \{\rho^{x}, \rho^{x-1}, \dots, \rho|\cdot|^y\},
\]
where $\rho \in \Cusp_\unit(\GL_d(F))$ and $x,y \in \R$ with $x-y \in \Z$ and $x \geq y$.
For a segment $[x,y]_\rho$ as above, 
we have a Steinberg representation $\Delta_\rho[x,y]$ of $\GL_{d(x-y+1)}(F)$, 
which is a unique irreducible subrepresentation of 
parabolically induced representation
\[
\rho|\cdot|^x \times \rho|\cdot|^{x-1} \times \dots \times \rho|\cdot|^y.
\]
\par

By the Langlands classification for $G_n$, 
any $\pi \in \Irr(G_n)$ is a unique irreducible subrepresentation 
of $\Delta_{\rho_1}[x_1,y_1] \times \dots \times \Delta_{\rho_r}[x_r,y_r] \rtimes \pi(\phi,\ep)$, 
where 
\begin{itemize}
\item
$\rho_1, \dots, \rho_r \in \Cusp_\unit(\GL_{d_i}(F))$; 
\item
$x_1+y_1 \leq \dots \leq x_r+y_r < 0$; 
\item
$\phi \in \Phi_\temp(G_{n_0})$ and $\ep \in \widehat{\Sc_\phi}$.
\end{itemize}
In this case, we write
\[
\pi = L(\Delta_{\rho_1}[x_1,y_1], \dots, \Delta_{\rho_r}[x_r,y_r]; \pi(\phi,\ep)),
\]
and call $(\Delta_{\rho_1}[x_1,y_1], \dots, \Delta_{\rho_r}[x_r,y_r]; \pi(\phi,\ep))$
the \emph{Langlands data} for $\pi$. 
\par

We say that 
an irreducible representation 
$\pi = L(\Delta_{\rho_1}[x_1,y_1], \dots, \Delta_{\rho_r}[x_r,y_r]; \pi(\phi,\ep))$
is \emph{of good parity}
if 
\begin{itemize}
\item
$x_1, \dots, x_r \in (1/2)\Z$;
\item
$\rho_i \boxtimes S_{2|x_i|+1}$ is self-dual representation of $W_F \times \SL_2(\C)$
of the same type as $\phi$ for $i = 1,\dots,r$; and
\item
$\phi \in \Phi_\gp(G_{n_0})$.
\end{itemize}
This definition is equivalent to the one in Section \ref{intro}.
We denote the set of equivalence classes of irreducible representations of $G_n$ of good parity
by $\Irr_\gp(G_n)$. 
Note that if $\psi \in \Psi_\gp(G_n)$, then $\Pi_\psi \subset \Irr_\gp(G_n)$.
Moreover, by Proposition \ref{bad} (together with \cite[Theorem 5.3]{At2}), 
Problem \ref{pro1} is reduced to the case where $\pi \in \Irr_\gp(G_n)$.

\subsection{Derivatives}\label{s.der}
As in \cite{At}, we use the following notion. 

\begin{defi}
Fix $\rho \in \Cusp(\GL_d(F))$. 
Let $\pi$ be a smooth representation of $G_n$ of finite length. 

\begin{enumerate}
\item
The \emph{$k$-th $\rho$-derivative} of $\pi$ is 
a semisimple representation $D_{\rho}^{(k)}(\pi)$ satisfying
\[
[\Jac_{P_{dk}}(\pi)] = \rho^k \otimes D_{\rho}^{(k)}(\pi) + \sum_{i \in I }\tau_i \otimes \pi_i, 
\]
where 
\begin{itemize}
\item
$[\Jac_{P_{dk}}(\pi)]$ is the semisimplification of Jacquet module of $\pi$
along the standard parabolic subgroup $P_{dk}$ with Levi part 
$\GL_{dk}(F) \times G_{n-dk}$; 
\item
$\rho^k = \rho \times \dots \times \rho$ ($k$-times);
\item
$\tau_i \in \Irr(\GL_{dk}(F))$ such that $\tau_i \not\cong \rho^k$.
\end{itemize}

\item
For simplicity, we write $D_{\rho}(\pi) = D_{\rho}^{(1)}(\pi)$.

\item
When $D_{\rho}^{(k)}(\pi) \not= 0$ but $D_{\rho}^{(k+1)}(\pi) = 0$, 
we call $D_{\rho}^{(k)}(\pi)$ the \emph{highest $\rho$-derivative} of $\pi$.
\end{enumerate}
\end{defi}

In this paper, we use derivatives as the following forms. 
\begin{defi}\label{d.d}
Fix $\rho \in \Cusp_\unit(\GL_d(F))$. 
Let $\pi \in \Irr(G_n)$. 
Define 
\begin{align*}
\DD^+_\rho(\pi) &= \left[(x,k_0),(x+1,k_1),\dots,(x+t-1,k_{t-1}); \pi^+\right], \\
\DD^-_\rho(\pi) &= \left[(x,k_0),(x-1,k_1),\dots,(x-t+1,k_{t-1}); \pi^-\right]
\end{align*}
as follows. 

\begin{enumerate}
\item
If $D_{\rho|\cdot|^z}(\pi) = 0$ for any $z \in \R$, 
we set $\DD^+_\rho(\pi) = \DD^-_\rho(\pi) = [\pi]$ 
(so that $t=0$ and $\pi^\pm = \pi$). 

\item
Otherwise, for $\epsilon \in \{\pm\}$, set
\[
x = \left\{
\begin{aligned}
\max&\{z \in \R \;|\; D_{\rho|\cdot|^z}(\pi) \not= 0\} \iif \epsilon = +, \\
\min&\{z \in \R \;|\; D_{\rho|\cdot|^z}(\pi) \not= 0\} \iif \epsilon = -.
\end{aligned}
\right. 
\]
Define $t > 0$ and $(k_0,\dots,k_{t-1})$ so that 
\[
\pi^\epsilon \coloneqq
D_{\rho|\cdot|^{x+\epsilon (t-1)}}^{(k_{t-1})} \circ \dots \circ D_{\rho|\cdot|^x}^{(k_0)}(\pi)
\not= 0,
\]
but for $j = 0, \dots, t$, we have 
\[
D_{\rho|\cdot|^{x+\epsilon j}}^{(k_{j}+1)} \circ
\left( D_{\rho|\cdot|^{x+\epsilon (j-1)}}^{(k_{j-1})} 
\circ \dots \circ D_{\rho|\cdot|^x}^{(k_0)}\right)(\pi) = 0
\]
with $k_t \coloneqq 0$.
\end{enumerate}
\end{defi}

Recall that if $\rho \in \Cusp^\bot(\GL_d(F))$ and $x \in \R$ with $x \not= 0$, 
for $\pi \in \Irr(G_n)$, 
its highest $\rho|\cdot|^x$-derivative $D_{\rho|\cdot|^x}^{(k)}(\pi)$ is also irreducible. 
Moreover, if we know the Langlands data for $\pi$, 
one can compute the Langlands data for $D_{\rho|\cdot|^x}^{(k)}(\pi)$ 
(see \cite[Sections 6, 7]{AM}).
In particular, for $\epsilon \in \{\pm\}$, if 
$\DD^\epsilon_\rho(\pi) = \left[(x,k_0),\dots,(x+\epsilon(t-1),k_{t-1}); \pi^\epsilon\right]$ 
with $\epsilon x > 0$, 
we can compute the Langlands data for $\pi^\epsilon$.

\section{Main results and Examples}\label{s.main}
In this section, we state our main results and give some examples. 

\subsection{Extended multi-segments}
To describe $A$-packets $\Pi_\psi$ for $\psi \in \Psi_\gp(G_n)$, 
in \cite{At}, we introduced the following notions.

\begin{defi}\label{segments}
\begin{enumerate}
\item
An \emph{extended segment} is a triple $([A,B]_\rho, l, \eta)$,
where
\begin{itemize}
\item
$[A,B]_\rho = \{\rho|\cdot|^A, \dots, \rho|\cdot|^B \}$ is a segment; 
\item
$l \in \Z$ with $0 \leq l \leq \half{b}$, where $b \coloneqq \#[A,B]_\rho = A-B+1$; 
\item
$\eta \in \{\pm1\}$. 
\end{itemize}

\item
An \emph{extended multi-segment} for $G_n$ is a weak equivalence class of multi-sets of extended segments 
\[
\EE = \bigcup_{\rho}\{ ([A_i,B_i]_{\rho}, l_i, \eta_i) \}_{i \in (I_\rho,>)}
\]
such that 
\begin{itemize}
\item
$\rho$ runs over $\sqcup_{d \geq 1}\Cusp^\bot(\GL_d(F))$;
\item
$I_\rho$ is a totally ordered finite set with a fixed order $>$ satisfying 
\[
A_i < A_j,\; B_i < B_j \implies i < j,
\]
which is called an \emph{admissible order};

\item
$A_i + B_i \geq 0$ for all $\rho$ and $i \in I_\rho$; 

\item
as a representation of $W_F \times \SL_2(\C) \times \SL_2(\C)$, 
\[
\psi_\EE \coloneqq 
\bigoplus_\rho \bigoplus_{i \in I_\rho} \rho \boxtimes S_{a_i} \boxtimes S_{b_i}
\]
belongs to $\Psi_\gp(G_n)$, 
where $a_i \coloneqq A_i+B_i+1$ and $b_i \coloneqq A_i-B_i+1$; 

\item
a sign condition
\[
\prod_{\rho} \prod_{i \in I_\rho} (-1)^{[\half{b_i}]+l_i} \eta_i^{b_i} = 1
\]
holds. 
\end{itemize}

\item
Two extended segments $([A,B]_\rho, l, \eta)$ and $([A',B']_{\rho'}, l', \eta')$ are \emph{equivalent} 
if 
\begin{itemize}
\item
$[A,B]_\rho = [A',B']_{\rho'}$; 
\item
$l = l'$; and
\item
$\eta = \eta'$ whenever $l = l' < \half{b}$. 
\end{itemize}
Similarly, $\EE = \cup_{\rho}\{ ([A_i,B_i]_{\rho}, l_i, \eta_i) \}_{i \in (I_\rho,>)}$ 
and $\EE' = \cup_{\rho}\{ ([A'_i,B'_i]_{\rho}, l'_i, \eta'_i) \}_{i \in (I_\rho,>)}$ are \emph{weak equivalent}
if $([A_i,B_i]_\rho, l_i, \eta_i)$ and $([A'_i,B'_i]_{\rho}, l'_i, \eta'_i)$ are equivalent for all $\rho$ and $i \in I_\rho$.

\item
An admissible order $>$ on $I_\rho$ is called \emph{very admissible} 
if it satisfies a stronger condition 
\[
B_i < B_j \implies i < j. 
\]
We say that \emph{$\EE$ has very admissible orders}
if the admissible order $>$ on $I_\rho$ is very admissible for any $\rho$.
\end{enumerate}
\end{defi}

In \cite{At}, to an extended multi-segment $\EE$ for $G_n$, 
we associate a representation $\pi(\EE)$ of $G_n$. 
It is irreducible or zero.
The following properties were proven in \cite[Theorems 1.2, 1.3, 1.4]{At}: 
\begin{thm}
\begin{enumerate}
\item
For $\psi = \oplus_{\rho}(\oplus_{i \in I_\rho} \rho \boxtimes S_{a_i} \boxtimes S_{b_i}) \in \Psi_\gp(G_n)$, 
after fixing a very admissible order $>$ on $I_\rho$ for each $\rho$, 
we have 
\[
\Pi_\psi = \{\pi(\EE) \;|\; \psi_\EE \cong \psi\} \setminus \{0\}. 
\]

\item
There exists a non-vanishing criterion for $\pi(\EE)$. 

\end{enumerate}
\end{thm}
\par

\subsection{Specifying of Arthur type representations}
Now we give an algorithm to determine 
whether given $\pi \in \Irr_\gp(G_n)$ is of Arthur type or not. 
The following is the first main result. 

\begin{alg}\label{main.alg}
Let $\pi \in \Irr_\gp(G_n)$.
Assume that the Langlands data for $\pi$ is given. 

\begin{description}
\item[Step $1^+$]
Suppose that 
there exist $\rho \in \Cusp^\bot(\GL_d(F))$ and $z \geq 1$ such that $D_{\rho|\cdot|^z}(\pi) \not= 0$. 
Write 
\[
\DD^+_\rho(\pi) = \left[(x,k_0),(x+1,k_1),\dots,(x+t-1,k_{t-1}); \pi^+ \right]
\]
with $\pi^+ \in \Irr_\gp(G_{n_+})$.
Note that $t > 0$ and $x \geq 1$. 
Then $\pi$ is of Arthur type if and only if 
there exists an extended multi-segment $\EE^+$ for $G_{n_+}$ such that 
\begin{itemize}
\item
$\pi^+ \cong \pi(\EE^+)$; 
\item
$\max\{B \;|\; ([A,B]_\rho,l,\eta) \in \EE^+\} = x-1$; 
\item
$\EE^+$ contains extended segments of the form 
$([x+j-2,x-1]_\rho,*,*)$ with at least $k_{j-1}-k_{j}$ times for $1 \leq j \leq t$ with $k_t \coloneqq 0$.
\end{itemize}
In this case, 
let $\EE$ be given from $\EE^+$ 
by replacing $([x+j-2,x-1]_\rho,l_j,\eta_j)$ with $([x+j-1,x]_\rho,l_j,\eta_j)$
exactly $k_{j-1}-k_{j}$ times for $1 \leq j \leq t$.
Then $\pi \cong \pi(\EE)$. 

\item[Step $1^-$]
Suppose that 
there exist $\rho \in \Cusp^\bot(\GL_d(F))$ and $z < 0$ such that $D_{\rho|\cdot|^z}(\pi) \not= 0$. 
Write 
\[
\DD^-_\rho(\pi) = \left[(x,k_0),(x-1,k_1),\dots,(x-t+1,k_{t-1}); \pi^- \right]
\]
with $\pi^- \in \Irr_\gp(G_{n_-})$.
Note that $t > 0$ and $x < 0$. 
Then $\pi$ is of Arthur type if and only if 
there exists an extended multi-segment $\EE^-$ for $G_{n_-}$ such that 
\begin{itemize}
\item
$\pi^- \cong \pi(\EE^-)$; 
\item
$\min\{B \;|\; ([A,B]_\rho,l,\eta) \in \EE^-\} = x+1$; 
\item
$\EE^-$ contains extended segments of the form 
$([-x+j-2,x+1]_\rho,*,*)$ with at least $k_{j-1}-k_{j}$ times for $1 \leq j \leq t$ with $k_t \coloneqq 0$.
Here, when $x = -1/2$, we omit this condition for $j = 1$ since $[-1/2,1/2]_\rho = \emptyset$. 
\end{itemize}
In this case, 
let $\EE$ be given from $\EE^-$ by 
replacing $([-x+j-2,x+1]_\rho,l_j,\eta_j)$ with $([-x+j-1,x]_\rho,l_j+1,\eta_j)$
exactly $k_{j-1}-k_{j}$ times for $1 \leq j \leq t$. 
Here, when $x=-1/2$, we understand this operation for $j=1$
as adding $([1/2,-1/2]_\rho,1,+1)$ exactly $k_0-k_1$ times 
whose indices are less than any element in $I_\rho$. 
Then $\pi \cong \pi(\EE)$.

\item[Step $2$]
Otherwise, i.e., 
suppose that for $\rho \in \Cusp^\bot(\GL_d(F))$ and $z \in \R$, 
\[
D_{\rho|\cdot|^z}(\pi) \not= 0 \implies z \in \{0,1/2\}.
\]
Write 
\[
\pi = L(\Delta_{\rho_1}[x_1,y_1], \dots, \Delta_{\rho_r}[x_r,y_r]; \pi(\phi,\ep))
\]
as in the Langlands classification. 
For $\rho \in \Cusp^\bot(\GL_d(F))$ and $z \in (1/2)\Z$ with $z \geq 0$, set 
\begin{align*}
k_{\rho,z} 
&\coloneqq 
\#\{i \in \{1,\dots,r\} \;|\; \text{$\rho_i \cong \rho$, and, $x_i = z$ or $y_i  =-z$}\}
+m_\phi(\rho \boxtimes S_{2z+1}), 
\end{align*}
where $m_\phi(\rho \boxtimes S_{a})$ is the multiplicity of $\rho \boxtimes S_a$ in $\phi$.
Then unless 
$k_{\rho,z} \geq k_{\rho, z+1}$ for any $\rho$ and $z \geq 0$, 
then $\pi$ is not of Arthur type. 
In this case, set 
\[
\psi = \bigoplus_\rho \bigoplus_{z \in (1/2)\Z_{\geq 0}} 
\left(\rho \boxtimes S_{z+1+\delta_z} \boxtimes S_{z+1-\delta_z}\right)^{\oplus (k_{\rho,z}-k_{\rho,z+1})}
\]
with $\delta_z \in \{0,1/2\}$ such that $z+\delta_z \in \Z$.
Then $\pi$ is of Arthur type if and only if $\pi \in \Pi_\psi$.
\end{description}
\end{alg}

The claims in this algorithm will be proven in Section \ref{proof.alg} below. 
To apply this algorithm, in Step $1^\pm$, 
we need to know 
all extended multi-segments $\EE^\pm$ such that $\pi^\pm \cong \pi(\EE^\pm)$. 
It is done by the result in the next subsection.
An example of Algorithm \ref{main.alg} is given in Section \ref{s.ex2}.

\subsection{Strongly equivalence classes}
Recall that 
for $\psi,\psi' \in \Psi_\gp(G_n)$, 
even if $\psi \not\cong \psi'$, 
the $A$-packets $\Pi_\psi$ and $\Pi_{\psi'}$ can have an intersection.
In other words, 
for two extended multi-segments $\EE$ and $\EE'$ for $G_n$, 
even if $\EE \not= \EE'$, one might have $\pi(\EE) \cong \pi(\EE') \not= 0$. 
In this subsection, we determine this situation.

\begin{defi}\label{CUIP}
Let $\EE_1$ and $\EE_2$ be two extended multi-segments for $G_n$. 
We say that $\EE_1$ and $\EE_2$ are \emph{strongly equivalent}
if $\EE_2$ can be obtained from $\EE_1$ by a finite chain of the following three operations and their inverses: 
Write
\[
\EE = \bigcup_{\rho}\{ ([A_i,B_i]_{\rho}, l_i, \eta_i) \}_{i \in (I_\rho,>)}
\]
and let $i<j$ be adjacent elements in $I_\rho$.

\begin{description}
\item[Changing admissible orders (C)]
Suppose that $[A_i,B_i]_\rho \subset [A_j,B_j]_\rho$. 
Let $>'$ be another admissible order on $I_\rho$ given from $>$ by changing $i >' j$. 
Then we define $\EE \mapsto \EE'$ by replacing $(I_\rho,>)$ with $(I_\rho,>')$
and 
\[
\{([A_i,B_i]_{\rho}, l_i, \eta_i), ([A_j,B_j]_{\rho}, l_j, \eta_j)\}
\]
with 
\[
\{([A_i,B_i]_{\rho}, l'_i, \eta'_i), ([A_j,B_j]_{\rho}, l'_j, \eta'_j)\}
\]
where $l_i', \eta_i', l_j', \eta_j'$ are given explicitly in \cite[Theorem 1.3]{X3} (see also \cite[Section 4.2]{At}).

\item[Union-Intersection (UI)]
Suppose that $B_i < B_j$, $A_i < A_j$ and 
that one of conditions (1)--(3) in \cite[Section 5.2]{At} holds. 
Then we define $\EE \mapsto \EE'$ by replacing 
\[
\{([A_i,B_i]_{\rho}, l_i, \eta_i), ([A_j,B_j]_{\rho}, l_j, \eta_j)\}
\]
with 
\[
\{([A_j,B_i]_{\rho}, l'_i, \eta'_i), ([A_i,B_j]_{\rho}, l'_j, \eta'_j)\}
\]
where $l_i', \eta_i', l_j', \eta_j'$ are given explicitly in \cite[Theorem 5.2]{At}.
Note that $[A_j,B_i]_{\rho} = [A_i,B_i]_{\rho} \cup [A_j,B_j]_{\rho}$
and $[A_i,B_j]_{\rho} = [A_i,B_i]_{\rho} \cap [A_j,B_j]_{\rho}$.
Here, when $B_j = A_i+1$ so that $[A_i,B_j]_{\rho} = \emptyset$, 
we remove $([A_j,B_i]_{\rho}, l'_j, \eta'_j)$.

\item[Phantom (dis)appearing (P)]
Formally add $([l-1,-l]_\rho, l, +1)$ for an integer $l > 0$
or $([l-1/2,-l-1/2]_\rho,l,+1)$ for an integer $l \geq 0$
to $\EE$, whose index $i_0$ is the minimum in $I_\rho \cup \{i_0\}$. 
\end{description}
\end{defi}

The second main result is as follows.
\begin{thm}\label{thm:s-eq}
Let $\EE_1$ and $\EE_2$ be two extended multi-segments for $G_n$. 
Suppose that $\pi(\EE_1) \not= 0$. 
Then 
$\pi(\EE_1) \cong \pi(\EE_2)$ if and only if $\EE_1$ and $\EE_2$ are strongly equivalent.
\end{thm}

This theorem is proven in Section \ref{proof.eq} below. 
By this theorem, if $\pi \in \Irr_\gp(G_n)$ is known to be of Arthur type, 
then one can list all $\psi \in \Psi_\gp(G_n)$ such that $\pi \in \Pi_\psi$. 

\subsection{Example of Theorem \ref{thm:s-eq}}\label{s.ex1}
In this and the next subsections, 
we set $\rho = \1_{\GL_1(F)}$ and drop $\rho$ from the notation.
Moreover, we only consider $A$-parameters of the form $\psi = \oplus_{i \in I} S_{a_i} \boxtimes S_{b_i}$
and extended multi-segments $\EE = \{ ([A_i,B_i], l_i, \eta_i) \}_{i \in (I,>)}$. 
When $\phi = S_{2x_1+1} + \dots + S_{2x_r+1}$ and 
$\ep(S_{2x_i+1}) = \ep_i \in \{\pm1\}$, 
we write $\pi(\phi,\ep) = \pi(x_1^{\ep_1},\dots,x_r^{\ep_r})$.
\par

As in \cite[Section 3.1]{At}, we regard $\EE$ as the following symbol.
When $\EE = \{([A,B],l,\eta)\}$ is a singleton, we write
\[
\EE = 
\left(
\begin{array}{rcl}
\underbrace{\overset{B}{\lhd} \lhd \cdots \overset{B+l-1}{\lhd}}_l 
&
\overset{B+l}{\odot} \odot \cdots \odot \overset{A-l} \odot 
&
\underbrace{\overset{A-l+1}{\rhd} \cdots \rhd \overset{A}{\rhd}}_l
\end{array}
\right),
\] 
where $\odot$ is replaced with $\oplus$ and $\ominus$ alternately, 
starting with $\oplus$ if $\eta = +1$ (\resp $\ominus$ if $\eta = -1$).
In general, we put each symbol vertically. 
\par

Now we consider $\pi = \pi(0^-,1^+,2^-) \in \Irr(\Sp_8(F))$. 
It is known that $\pi$ is supercuspidal. 
Let us construct all $A$-parameters $\psi \in \Psi_\gp(\Sp_8(F))$ such that $\pi \in \Pi_\psi$.

\begin{enumerate}
\item
Define 
\[
\EE_1 = 
\bordermatrix{
& 0 & 1 & 2 \cr
& \ominus && \cr
& & \oplus & \cr
& & & \ominus
}.
\]
Clearly, $\pi(\EE_1) = \pi$. 
The associated $A$-parameter is $\psi_1 = \psi_{\EE_1} = S_1 + S_3 + S_5 \in \Phi_\gp(\Sp_8(F))$. 

\item
By using (UI) for the first and second lines of $\EE_1$, we obtain 
\[
\EE_2 = 
\bordermatrix{
& 0 & 1 & 2 \cr
& \ominus & \oplus & \cr
& & & \ominus
}.
\]
The associated $A$-parameter is $\psi_2 = \psi_{\EE_2} = S_2 \boxtimes S_2 + S_5$.

\item
After adding $([0,-1],1,1)$ to $\EE_2$ by (P), 
we use (UI). 
Then we obtain 
\[
\EE_3 = 
\bordermatrix{
& -1 & 0 & 1 & 2 \cr
& \lhd & \ominus & \rhd & \cr
&& \oplus && \cr
&& & & \ominus
}.
\]
The associated $A$-parameter is $\psi_3 = \psi_{\EE_3} = S_1 \boxtimes S_3 + S_1 + S_5$.

\item
By using (UI) for the second and third lines of $\EE_1$, we obtain 
\[
\EE_4 = 
\bordermatrix{
& 0 & 1 & 2 \cr
& \ominus && \cr
& & \oplus & \ominus
}.
\]
The associated $A$-parameter is $\psi_4 = \psi_{\EE_4} = S_1 + S_4 \boxtimes S_2$.

\item
By using (UI) for $\EE_2$ or $\EE_4$, 
we obtain 
\[
\EE_5 = 
\bordermatrix{
& 0 & 1 & 2 \cr
& \ominus & \oplus & \ominus
}.
\]
The associated $A$-parameter is $\psi_5 = \psi_{\EE_5} = S_3 \boxtimes S_3$.

\item
After adding $([0,-1],1,1)$ to $\EE_5$ by (P), 
we use (UI). 
Then we obtain 
\[
\EE_6 = 
\bordermatrix{
& -1 & 0 & 1 & 2 \cr
& \lhd & \ominus & \oplus & \rhd \cr
& & \ominus &&
}.
\]
The associated $A$-parameter is $\psi_6 = \psi_{\EE_6} = S_2 \boxtimes S_4 + S_1$.

\item[(8)]
After adding $([1,-2],2,1)$ to $\EE_5$ by (P), 
we use (UI). 
Then we obtain 
\[
\EE_8 = 
\bordermatrix{
& -2 & -1 & 0 & 1 & 2 \cr
& \lhd & \lhd & \ominus & \rhd & \rhd \cr
& & & \oplus & \ominus &
}.
\]
The associated $A$-parameter is $\psi_8 = \psi_{\EE_8} = S_1 \boxtimes S_5 + S_2 \boxtimes S_2$.

\item[(7)]
By using the inverse of (UI) for the second line of $\EE_8$, we obtain 
\[
\EE_7 = 
\bordermatrix{
& -2 & -1 & 0 & 1 & 2 \cr
& \lhd & \lhd & \ominus & \rhd & \rhd \cr
& & & \oplus & & \cr
& & & & \ominus &
}.
\]
The associated $A$-parameter is $\psi_7 = \psi_{\EE_7} = S_1 \boxtimes S_5 + S_1 + S_3$.

\item[(9)]
After adding $([1,-2],2,1)$ to $\EE_6$ by (P), 
we use (UI). 
Then we obtain 
\[
\EE_9 = 
\bordermatrix{
& -2 & -1 & 0 & 1 & 2 \cr
& \lhd & \lhd & \ominus & \rhd & \rhd \cr
&  & \lhd & \oplus & \rhd &  \cr
& & & \ominus &&
}.
\]
The associated $A$-parameter is $\psi_9 = \psi_{\EE_9} = S_1 \boxtimes S_5 + S_1 \boxtimes S_3 + S_1$.
Note that we can obtain $\EE_9$ from $\EE_8$ by using (P), (C) and (UI).
\end{enumerate}

By Theorem \ref{thm:s-eq}, 
we can check that $\EE_1, \dots, \EE_9$ 
are all of the extended multi-segments $\EE$ such that $\pi(\EE) \cong \pi$.
The relation among $\{\EE_1,\dots,\EE_9\}$ can be written as follows:
\[
\xymatrix{
\EE_1 \ar@{-}[r] \ar@{-}[d] & \EE_2 \ar@{-}[r] \ar@{-}[d] & \EE_3 \\
\EE_4 \ar@{-}[r] & \EE_5 \ar@{-}[r] \ar@{-}[d] & \EE_6 \ar@{-}[d] \\
\EE_7 \ar@{-}[r] & \EE_8 \ar@{-}[r] & \EE_9
}
\]
\par

In \cite[Definition 6.1, Theorem 6.2]{At}, we defined 
an explicit map $\EE \mapsto \hat\EE$ 
such that $\pi(\hat\EE)$ is the Aubert dual of $\pi(\EE)$.
Since $\pi$ is supercuspidal, 
it is fixed by the Aubert duality. 
Hence the set $\{\EE_1,\dots,\EE_9\}$ is stable under $\EE \mapsto \hat\EE$. 
Indeed, it is easy to check that $\hat\EE_i = \EE_{10-i}$ for $1 \leq i \leq 9$.

\subsection{Example of Algorithm \ref{main.alg}}\label{s.ex2}
Here, we give an example of applying Algorithm \ref{main.alg}. 
\par

For $\ep = (\ep_0,\ep_1,\ep_2) \in \{\pm\}^3$ with $\ep_0\ep_1\ep_2 = 1$, 
consider
\[
\textstyle
\pi^\ep = L\left(\Delta[-\half{1},-\half{5}], |\cdot|^{-\half{1}}, \Delta[\half{3},-\half{5}]; 
\pi((\half{1})^{\ep_1},(\half{3})^{\ep_2},(\half{5})^{\ep_3})\right) \in \Irr(\SO_{31}(F)). 
\]
Then 
\[
\Pi = \left\{
\pi^\ep
\;\middle|\; \ep = (\ep_0,\ep_1,\ep_2) \in \{\pm\}^3,\; \ep_0\ep_1\ep_2 = 1
\right\}
\]
is a non-tempered $L$-packet of $\SO_{31}(F)$.
Let us determine whether $\pi_\ep$ is of Arthur type or not. 
\par

First of all, we apply Step $1^-$ of Algorithm \ref{main.alg}. 
Then $\DD^-(\pi^\ep) = [(-\half{1},2),(-\half{3},1),(-\half{5},1); \pi_1^\ep]$ 
with 
\[
\textstyle
\pi_1^\ep = L\left(\Delta[\half{3},-\half{5}]; \pi((\half{1})^{\ep_1},(\half{3})^{\ep_2},(\half{5})^{\ep_3})\right). 
\]
Next, we apply Step $1^-$ of Algorithm \ref{main.alg}. 

\begin{enumerate}
\item
Suppose that $\ep = (+,+,+)$. 
Then $\DD^+(\pi_1^{(+,+,+)}) = [(\half{3},2),(\half{5},2); \pi_2^{(+,+,+)}]$
with
\[
\textstyle
\pi_2^{(+,+,+)} = L\left(\Delta[\half{1},-\half{3}] ; \pi((\half{1})^{+},(\half{1})^{+},(\half{3})^{+})\right).
\]
Note that $\pi_2^{(+,+,+)}$ is in the situation of Step $2$ in Algorithm \ref{main.alg}.
According to this step, consider 
\[
\psi = S_2 \boxtimes S_1 + S_3\boxtimes S_2 + S_3 \boxtimes S_2.
\]
Then one can check that $\pi_2^{(+,+,+)} \in \Pi_{\psi}$. 
In fact, we have
\[
\pi_2^{(+,+,+)} \cong 
\pi
\bordermatrix{
& 1/2 & 3/2  \cr
& \oplus & \cr
& \oplus & \ominus \cr
& \ominus & \oplus 
}. 
\]
By Step $1^+$ in Algorithm \ref{main.alg}, we have
\[
\pi_1^{(+,+,+)} \cong 
\pi
\bordermatrix{
& 1/2 & 3/2 & 5/2  \cr
& \oplus & & \cr
& & \oplus & \ominus \cr
& & \ominus & \oplus 
}
\cong 
\pi
\bordermatrix{
& 1/2 & 3/2 & 5/2  \cr
& \oplus & & \cr
& & \oplus & \ominus \cr
& & \ominus & \cr
& & & \oplus 
}.
\]
Let $\EE_1$ and $\EE_2$ be the above extended multi-segments 
so that $\pi_1^{(+,+,+)} \cong \pi(\EE_1) \cong \pi(\EE_2)$. 
By Theorem \ref{thm:s-eq}, we see that 
if $\EE$ satisfies that $\pi_1^{(+,+,+)} \cong \pi(\EE)$, 
then $\EE \in \{\EE_1, \EE_2\}$, or $\EE$ is given from $\EE_2$ by (C). 
Since both $\EE_1$ and $\EE_2$ do not satisfy the conditions of Step $1^-$ in Algorithm \ref{main.alg}, 
we conclude that $\pi^{(+,+,+)}$ is not of Arthur type. 

\item
Suppose that $\ep = (-,-,+)$. 
Then $\DD^+(\pi_1^{(-,-,+)}) = [(\half{5},1); \pi_2^{(-,-,+)}]$
with
\[
\textstyle
\pi_2^{(-,-,+)} = \pi((\half{1})^{-},(\half{3})^{-},(\half{3})^{-},(\half{3})^{-},(\half{5})^{+}). 
\]
Moreover, $\DD^+(\pi_2^{(-,-,+)}) = [(\half{3},3), (\half{5},1); \pi_3^{(-,-,+)}]$
with 
\[
\textstyle
\pi_3^{(-,-,+)} = \pi((\half{1})^{-},(\half{1})^{-},(\half{1})^{-},(\half{1})^{-},(\half{3})^{+}). 
\]
Note that $\pi_3^{(-,-,+)}$ is in the situation of Step $2$ in Algorithm \ref{main.alg}.
According to this step, consider 
\[
\psi = (S_2 \boxtimes S_1)^{\oplus 3} + S_3 \boxtimes S_2.
\]
Then one can check that $\pi_3^{(-,-,+)} \in \Pi_{\psi}$. 
In fact, we have
\[
\pi_3^{(-,-,+)} \cong 
\pi
\bordermatrix{
& 1/2 & 3/2  \cr
& \ominus & \cr
& \ominus & \cr
& \ominus & \cr
& \ominus & \oplus 
}. 
\]
By Step $1^+$ in Algorithm \ref{main.alg}, we have
\begin{align*}
\pi_2^{(-,-,+)} 
&\cong 
\pi
\bordermatrix{
& 1/2 & 3/2 & 5/2 \cr
& \ominus && \cr
&& \ominus & \cr
&& \ominus & \cr
&& \ominus & \oplus 
}
\cong 
\pi
\bordermatrix{
& 1/2 & 3/2 & 5/2 \cr
& \ominus && \cr
&& \ominus & \cr
&& \lhd & \rhd \cr
&& \oplus &
}, 
\\
\pi_1^{(-,-,+)} 
&\cong 
\pi
\bordermatrix{
& 1/2 & 3/2 & 5/2 \cr
& \ominus && \cr
&& \ominus & \cr
&& \lhd & \rhd \cr
&&& \oplus 
}
\cong 
\pi
\bordermatrix{
&-1/2& 1/2 & 3/2 & 5/2 \cr
& \oplus & \ominus && \cr
&&& \ominus & \cr
&&& \lhd & \rhd \cr
&&&& \oplus 
}.
\end{align*}
Let $\EE_1$ and $\EE_2$ be the above extended multi-segments 
so that $\pi_1^{(-,-,+)} \cong \pi(\EE_1) \cong \pi(\EE_2)$. 
By Theorem \ref{thm:s-eq}, we see that 
if $\EE$ satisfies that $\pi_1^{(-,-,+)} \cong \pi(\EE)$, 
then $\EE \in \{\EE_1, \EE_2\}$, or $\EE$ is given from $\EE_1$ or $\EE_2$ by (C). 
Since both $\EE_1$ and $\EE_2$ do not satisfy the conditions of Step $1^-$ in Algorithm \ref{main.alg}, 
we conclude that $\pi^{(-,-,+)}$ is not of Arthur type. 

\item
Suppose that $\ep = (-,+,-)$. 
Then $\DD^+(\pi_1^{(-,+,-)}) = [(\half{5},1); \pi_2^{(-,+,-)}]$ 
with
\[
\textstyle
\pi_2^{(-,+,-)} = \pi((\half{1})^{-},(\half{3})^{+},(\half{3})^{+},(\half{3})^{+},(\half{5})^{-}). 
\]
Moreover, $\DD^+(\pi_2^{(-,+,-)}) = [(\half{3},2); \pi_3^{(-,+,-)}]$ 
with
\[
\textstyle
\pi_3^{(-,+,-)} = \pi((\half{1})^{-},(\half{1})^{-},(\half{1})^{-},(\half{3})^{+},(\half{5})^{-}). 
\]
Note that $\pi_3^{(-,+,-)}$ is in the situation of Step $2$ in Algorithm \ref{main.alg}.
According to this step, consider 
\[
\psi = (S_2 \boxtimes S_1)^{\oplus 2} + S_4 \boxtimes S_3.
\]
Then one can check that $\pi_3^{(-,+,-)} \in \Pi_{\psi}$. 
In fact, we have
\[
\pi_3^{(-,+,-)} \cong 
\pi
\bordermatrix{
& 1/2 & 3/2 & 5/2 \cr
& \ominus && \cr
& \ominus && \cr
& \ominus & \oplus & \ominus 
}
\cong 
\pi
\bordermatrix{
& 1/2 & 3/2 & 5/2 \cr
& \ominus & \oplus & \ominus \cr
& \ominus && \cr
& \ominus && 
}.
\]
By Step $1^+$ in Algorithm \ref{main.alg}, we have
\[
\pi_2^{(-,+,-)} \cong 
\pi
\bordermatrix{
& 1/2 & 3/2 & 5/2 \cr
& \ominus & \oplus & \ominus \cr
&& \ominus & \cr
&& \ominus &
},
\quad
\pi_1^{(-,+,-)} \cong 
\pi
\bordermatrix{
& 1/2 & 3/2 & 5/2 \cr
& \ominus & \oplus & \ominus \cr
&& \ominus & \cr
&&& \ominus 
}. 
\]
By the inverse of (UI), we have 
\[
\pi_1^{(-,+,-)} \cong 
\pi
\bordermatrix{
& 1/2 & 3/2 & 5/2 \cr
& \ominus & \oplus & \cr
&& \lhd & \rhd \cr
&&& \ominus 
}. 
\]
It satisfies the conditions of Step $1^-$ in Algorithm \ref{main.alg}. 
Hence
\[
\pi^{(-,+,-)} \cong 
\pi
\bordermatrix{
& -1/2 & 1/2 & 3/2 & 5/2 \cr
& \lhd & \rhd && \cr
& \lhd & \ominus & \oplus & \rhd \cr
&&& \lhd & \rhd \cr
&&&& \ominus 
}. 
\]
We conclude that $\pi^{(-,+,-)}$ is of Arthur type. 

\item
Suppose that $\ep = (+,-,-)$. 
Then $\DD^+(\pi_1^{(+,-,-)}) = [(\half{3},1),(\half{5},1); \pi_2^{(+,-,-)}]$ 
with
\[
\textstyle
\pi_2^{(+,+,+)} = L\left(\Delta[\half{1},-\half{5}] ; \pi((\half{1})^{+},(\half{3})^{-},(\half{3})^{-})\right).
\]
Note that $\pi_2^{(+,-,-)}$ is in the situation of Step $2$ in Algorithm \ref{main.alg}.
According to this step, consider 
\[
\psi = S_3 \boxtimes S_2 + S_4 \boxtimes S_3.
\]
Then one can check that $\pi_2^{(+,-,-)} \in \Pi_{\psi}$. 
In fact, we have
\[
\pi_2^{(+,-,-)} \cong 
\pi
\bordermatrix{
& 1/2 & 3/2 & 5/2 \cr
& \oplus & \ominus & \cr
& \lhd & \ominus & \rhd 
}
\cong 
\pi
\bordermatrix{
& 1/2 & 3/2 & 5/2 \cr
& \oplus & \ominus & \oplus \cr
& \oplus & \ominus &  
}.
\]
By Step $1^+$ in Algorithm \ref{main.alg}, we have
\[
\pi_1^{(+,-,-)} \cong 
\pi
\bordermatrix{
& 1/2 & 3/2 & 5/2 \cr
& \oplus & \ominus & \oplus \cr
& & \oplus & \ominus &  
}
\cong
\pi
\bordermatrix{
& 1/2 & 3/2 & 5/2 \cr
& \oplus & \ominus & \cr
& & \lhd & \rhd \cr
& & & \ominus 
}.
\]
It satisfies the conditions of Step $1^-$ in Algorithm \ref{main.alg}. 
Hence
\[
\pi^{(+,-,-)} \cong 
\pi
\bordermatrix{
& -1/2 & 1/2 & 3/2 & 5/2 \cr
& \lhd & \rhd && \cr
& \lhd & \oplus & \ominus & \rhd \cr
&&& \lhd & \rhd \cr
&&&& \ominus 
}. 
\]
We conclude that $\pi^{(+,-,-)}$ is of Arthur type. 
\end{enumerate}
In conclusion, the non-tempered $L$-packet $\Pi$ of $\SO_{31}(F)$
has exactly two representations of Arthur type, 
and has exactly two representations which are not of Arthur type. 

\if()
For $\ep = (\ep_0,\ep_1,\ep_2) \in \{\pm\}^3$ with $\ep_0\ep_1\ep_2 = 1$, 
consider
\[
\pi^\ep = L(\Delta[0,-3], \Delta[2,-3]; \pi(0^{\ep_0},1^{\ep_1},2^{\ep_2})) \in \Irr(\Sp_{28}(F)). 
\]
Then 
\[
\Pi = \left\{
\pi^\ep
\;\middle|\; \ep = (\ep_0,\ep_1,\ep_2) \in \{\pm\}^3,\; \ep_0\ep_1\ep_2 = 1
\right\}
\]
is a non-tempered $L$-packet of $\Sp_{28}(F)$.
Let us determine whether $\pi_\ep$ is of Arthur type or not. 

\begin{enumerate}
\item
Suppose that $\ep = (+,+,+)$. 
Then $\DD^+(\pi^{(+,+,+)}) = [(2,2),(3,2); \pi_1^{(+,+,+)}]$
with
\[
\pi_1^{(+,+,+)} = L(\Delta[0,-2], \Delta[1,-2]; \pi(0^+,1^+,1^+)).
\]
Moreover, $\DD^+(\pi_1^{(+,+,+)}) = [(1,2),(2,1); \pi_2^{(+,+,+)}]$
with
\[
\pi_2^{(+,+,+)} = L(\Delta[0,-2], \Delta[0,-1]; \pi(0^+,0^+,1^+)).
\]
Note that $\pi_2^{(+,+,+)}$ is in the situation of Step $2$ in Algorithm \ref{main.alg}.
According to this step, consider 
\[
\psi = S_1 + S_1 + S_2\boxtimes S_2 + S_3 \boxtimes S_3.
\]
Then one can check that $\pi_2^{(+,+,+)} \in \Pi_{\psi}$. 
In fact, we have
\[
\pi_2^{(+,+,+)} \cong 
\pi
\bordermatrix{
& 0 & 1 & 2 \cr
& \oplus & & \cr
& \oplus & & \cr
& \lhd & \rhd & \cr
& \lhd & \oplus & \rhd
}
\cong 
\pi
\bordermatrix{
& 0 & 1 & 2 \cr
& \oplus & & \cr
& \lhd & \oplus & \rhd \cr
& \oplus & & \cr
& \lhd & \rhd & 
}.
\]
By Step $1^+$ in Algorithm \ref{main.alg}, we have
\[
\pi_1^{(+,+,+)} \cong 
\pi
\bordermatrix{
& 0 & 1 & 2 \cr
& \oplus & & \cr
& \lhd & \oplus & \rhd \cr
&& \oplus & \cr
&& \lhd & \rhd 
}.
\]
Let $\EE_1$ be this extended multi-segment so that $\pi_1^{(+,+,+)} \cong \pi(\EE_1)$. 
By Theorem \ref{thm:s-eq}, we see that 
there is no $\EE_1' \not= \EE_1$ such that $\pi_1^{(+,+,+)} \cong \pi(\EE_1')$.
Since $\EE_1$ does not satisfy the conditions of Step $1^+$ in Algorithm \ref{main.alg}, 
we conclude that $\pi^{(+,+,+)}$ is not of Arthur type. 

\item
Suppose that $\ep = (-,-,+)$. 
Then $\DD^+(\pi^{(-,-,+)}) = [(2,1),(3,1); \pi_1^{(-,-,+)}]$ 
with
\[
\pi_1^{(-,-,+)} = L(\Delta[0,-2], \Delta[1,-3]; \pi(0^-,1^-,2^+)).
\]
Moreover, $\DD^+(\pi_1^{(-,-,+)}) = [(1,1),(2,1); \pi_2^{(-,-,+)}]$
with
\[
\pi_2^{(-,-,+)} = L(\Delta[0,-3], \Delta[0,-1]; \pi(0^-,1^-,2^+)).
\]
Note that $\pi_2^{(-,-,+)}$ is in the situation of Step $2$ in Algorithm \ref{main.alg}.
According to this step, consider 
\[
\psi = S_1 + S_2\boxtimes S_2 + S_4 \boxtimes S_4.
\]
Then one can check that $\pi_2^{(-,-,+)} \in \Pi_{\psi}$. 
In fact, we have
\[
\pi_2^{(-,-,+)} \cong 
\pi
\bordermatrix{
& 0 & 1 & 2 & 3 \cr
& \ominus & & &\cr
& \lhd & \rhd & &\cr
& \lhd & \ominus & \oplus & \rhd 
}
\cong 
\pi
\bordermatrix{
& 0 & 1 & 2 & 3 \cr
& \ominus & & &\cr
& \lhd & \ominus & \oplus & \rhd \cr
& \lhd & \rhd & &
}.
\]
By Step $1^+$ in Algorithm \ref{main.alg}, we have
\[
\pi_1^{(-,-,+)} \cong 
\pi
\bordermatrix{
& 0 & 1 & 2 & 3 \cr
& \ominus & & &\cr
& \lhd & \ominus & \oplus & \rhd \cr
& & \lhd & \rhd &
}, 
\quad
\pi^{(-,-,+)} \cong 
\pi
\bordermatrix{
& 0 & 1 & 2 & 3 \cr
& \ominus & & &\cr
& \lhd & \ominus & \oplus & \rhd \cr
& & & \lhd & \rhd
}. 
\]
We conclude that $\pi^{(-,-,+)}$ is of Arthur type. 

\item
Suppose that $\ep = (-,+,-)$. 
Then $\DD^+(\pi^{(-,+,-)}) = [(2,1),(3,1); \pi_1^{(-,+,-)}]$ 
with
\[
\pi_1^{(-,+,-)} = L(\Delta[0,-2], \Delta[1,-3]; \pi(0^-,1^+,2^-)).
\]
Moreover, $\DD^+(\pi_1^{(-,+,-)}) = [(1,1),(2,1); \pi_2^{(-,+,-)}]$ 
with
\[
\pi_2^{(-,+,-)} = L(\Delta[0,-3], \Delta[0,-1]; \pi(0^-,1^+,2^-)).
\]
Note that $\pi_2^{(-,+,-)}$ is in the situation of Step $2$ in Algorithm \ref{main.alg}.
According to this step, consider 
\[
\psi = S_1 + S_2\boxtimes S_2 + S_4 \boxtimes S_4.
\]
Then one can check that $\pi_2^{(-,+,-)} \in \Pi_{\psi}$. 
In fact, we have
\[
\pi_2^{(-,+,-)} \cong 
\pi
\bordermatrix{
& 0 & 1 & 2 & 3 \cr
& \ominus & & &\cr
& \lhd & \rhd & &\cr
& \lhd & \oplus & \ominus & \rhd 
}
\cong 
\pi
\bordermatrix{
& 0 & 1 & 2 & 3 \cr
& \ominus & & &\cr
& \lhd & \oplus & \ominus & \rhd \cr
& \lhd & \rhd & &
}.
\]
By Step $1^+$ in Algorithm \ref{main.alg}, we have
\[
\pi_1^{(-,+,-)} \cong 
\pi
\bordermatrix{
& 0 & 1 & 2 & 3 \cr
& \ominus & & &\cr
& \lhd & \oplus & \ominus & \rhd \cr
& & \lhd & \rhd &
}, 
\quad
\pi^{(-,+,-)} \cong 
\pi
\bordermatrix{
& 0 & 1 & 2 & 3 \cr
& \ominus & & &\cr
& \lhd & \oplus & \ominus & \rhd \cr
& & & \lhd & \rhd
}. 
\]
We conclude that $\pi^{(-,+,-)}$ is of Arthur type. 

\item
Suppose that $\ep = (+,-,-)$. 
Then $\DD^+(\pi^{(+,-,-)}) = [(2,2),(3,2); \pi_1^{(+,-,-)}]$ 
with
\[
\pi_1^{(+,-,-)} = L(\Delta[0,-2], \Delta[1,-2]; \pi(0^+,1^-,1^-)).
\]
Moreover, $\DD^+(\pi_1^{(+,-,-)}) = [(1,1); \pi_2^{(+,-,-)}]$ 
with
\[
\pi_2^{(+,-,-)} = L(\Delta[0,-2]^2; \pi(0^+,1^-,1^-)).
\]
Note that $\pi_2^{(+,-,-)}$ is in the situation of Step $2$ in Algorithm \ref{main.alg}.
According to this step, consider 
\[
\psi = S_1 + S_3 \boxtimes S_3 + S_3 \boxtimes S_3.
\]
Then one can check that $\pi_2^{(+,-,-)} \in \Pi_{\psi}$. 
In fact, we have
\[
\pi_2^{(+,-,-)} \cong 
\pi
\bordermatrix{
& 0 & 1 & 2 \cr
& \oplus & & \cr
& \lhd & \ominus & \rhd \cr
& \lhd & \ominus & \rhd
}
\cong 
\pi
\bordermatrix{
& 0 & 1 & 2 \cr
& \oplus & \ominus & \oplus \cr
& \oplus & \ominus & \oplus \cr
& \oplus & & 
}.
\]
By Step $1^+$ in Algorithm \ref{main.alg}, we have
\[
\pi_1^{(+,-,-)} \cong 
\pi
\bordermatrix{
& 0 & 1 & 2 \cr
& \oplus & \ominus & \oplus \cr
& \oplus & \ominus & \oplus \cr
& & \oplus &
}.
\]
Let $\EE_1$ be this extended multi-segment so that $\pi_1^{(+,-,-)} \cong \pi(\EE_1)$. 
By Theorem \ref{thm:s-eq}, we see that 
there are exactly 14 extended multi-segments $\EE_1, \dots, \EE_{14}$ 
such that $\pi(\EE_i) \cong \pi_1^{(+,-,-)}$ for $1 \leq i \leq 14$. 
The associated $A$-parameters $\psi_i = \psi_{\EE_i}$ are listed as follows: 
\begin{align*}
\psi_1 &= S_3 \boxtimes S_3 + S_3 \boxtimes S_3 + S_3, \\
\psi_2 &= S_2 \boxtimes S_4 + S_1 + S_3 \boxtimes S_3 + S_3, \\
\psi_3 &= S_1 \boxtimes S_5 + S_2 \boxtimes S_2 + S_3 \boxtimes S_3 + S_3, \\
\psi_4 &= S_1 + S_3 \boxtimes S_3 + S_3 + S_4 \boxtimes S_2, \\
\psi_5 &= S_1 \boxtimes S_5 + S_1 \boxtimes S_3 + S_1 + S_3 \boxtimes S_3 + S_3,\\
\psi_6 &= S_2 \boxtimes S_4 + S_1 + S_1 + S_3 + S_4 \boxtimes S_2, \\
\psi_7 &= S_1 \boxtimes S_5 + S_1 + S_3 \boxtimes S_3 + S_3 + S_3, \\
\psi_8 &= S_1 \boxtimes S_5 + S_1 + S_2 \boxtimes S_2 + S_3 + S_4 \boxtimes S_2, \\
\psi_9 &= S_2 \boxtimes S_2 + S_3 \boxtimes S_3 + S_4 \boxtimes S_2, \\
\psi_{10} &= S_1 \boxtimes S_5 + S_1 \boxtimes S_3 + S_1 + S_1 + S_3 + S_4 \boxtimes S_2, \\
\psi_{11} &= S_2 \boxtimes S_4 + S_1 + S_2 \boxtimes S_2 + S_4 \boxtimes S_2, \\
\psi_{12} &= S_1 \boxtimes S_5 + S_2 \boxtimes S_2 + S_2 \boxtimes S_2 + S_4 \boxtimes S_2, \\
\psi_{13} &= S_1 \boxtimes S_3 + S_1 + S_3 \boxtimes S_3 + S_4 \boxtimes S_2, \\
\psi_{14} &= S_1 \boxtimes S_3 + S_1 \boxtimes S_3 + S_1 + S_2 \boxtimes S_2 + S_4 \boxtimes S_2. 
\end{align*}
However, there are no $\EE_i$ satisfying the conditions of Step $1^+$ in Algorithm \ref{main.alg}. 
Hence we conclude that $\pi^{(+,-,-)}$ is not of Arthur type. 
\end{enumerate}
In conclusion, the non-tempered $L$-packet $\Pi$ of $\Sp_{28}(F)$
has exactly two representations of Arthur type, 
and has exactly two representations which are not of Arthur type. 
\fi

\section{Proof of main results}\label{s.proof}
Now we prove the claims in Algorithm \ref{main.alg} and Theorem \ref{thm:s-eq}. 

\subsection{Derivatives of $\pi(\EE)$}
For an irreducible representation $\pi$ of $G_n$ and for $\epsilon \in \{\pm\}$, 
we have defined $\DD^\epsilon_\rho(\pi)$ in Definition \ref{d.d}. 
The description of $\DD^\epsilon_\rho(\pi(\EE))$ was given in \cite[Section 5]{At}. 
\begin{thm}\label{thm5}
Let $\EE = \cup_{\rho}\{ ([A_i,B_i]_{\rho}, l_i, \eta_i) \}_{i \in (I_\rho,>)}$
be an extended multi-segment for $G_n$
such that $\pi(\EE) \not= 0$. 
Suppose that the order $>$ on $I_\rho$ is very admissible for any $\rho$.

\begin{enumerate}
\item
Write 
\[
\DD^+_\rho(\pi(\EE)) = \left[(x,k_0),(x+1,k_1),\dots,(x+t-1,k_{t-1}); \pi^+ \right]. 
\]
Assume that $t > 0$ and $x \geq 1$.
Then we can construct $\EE^*$ from $\EE$ by a finite chain of 
operations (C), (UI) and their inverses such that 
\begin{itemize}
\item
$\pi(\EE^*) \cong \pi(\EE)$;
\item
$\max\{ B \;|\; ([A,B]_\rho,l,\eta) \in \EE^*\} = x$; 
\item
the multi-set $\{ [A,B]_\rho \;|\; ([A,B]_\rho, l,\eta) \in \EE^*, \; B = x\}$ 
is exactly equal to 
\[
\bigsqcup_{j=1}^t \{ \underbrace{[x+j-1,x]_\rho, \dots, [x+j-1,x]_\rho}_{k_{j-1}-k_j} \}
\]
with $k_t \coloneqq 0$;
\item
if we define $\EE^+$ from $\EE^*$ by replacing 
each $([x+j-1,x]_\rho, l_j,\eta_j)$ with $([x+j-2,x-1]_\rho, l_j,\eta_j)$ for $1 \leq  j \leq t$, 
then $\pi^+ \cong \pi(\EE^+)$.
\end{itemize}

\item
Write 
\[
\DD^-_\rho(\pi(\EE)) = \left[(x,k_0),(x-1,k_1),\dots,(x-t+1,k_{t-1}); \pi^- \right]. 
\]
Assume that $t > 0$ and $x < 0$.
Then we can construct $\EE^*$ from $\EE$ by a finite chain of 
operations (C), (UI), (P) and their inverses such that 
\begin{itemize}
\item
$\pi(\EE^*) \cong \pi(\EE)$;
\item
$\min\{ B \;|\; ([A,B]_\rho,l,\eta) \in \EE^*\} = x$; 
\item
the multi-set $\{ [A,B]_\rho \;|\; ([A,B]_\rho, l,\eta) \in \EE^*, \; B = x\}$ 
is exactly equal to 
\[
\bigsqcup_{j=1}^t \{ \underbrace{[-x+j-1,x]_\rho, \dots, [-x+j-1,x]_\rho}_{k_{j-1}-k_j} \}
\]
with $k_t \coloneqq 0$;
\item
if we define $\EE^-$ from $\EE^*$ by replacing 
each $([-x+j-1,x]_\rho, l_j,\eta_j)$ with $([-x+j-2,x+1]_\rho, l_j-1,\eta_j)$ for $1 \leq  j \leq t$, 
then $\pi^- \cong \pi(\EE^-)$.
\end{itemize}
\end{enumerate}
\end{thm}
\begin{proof}
This is Theorem 5.1 together with Algorithms 5.5 and 5.6 in \cite{At}.
\end{proof}

\subsection{Proof of Algorithm \ref{main.alg}}\label{proof.alg}
First, we show the claims in Algorithm \ref{main.alg}. 

\begin{proof}[Proof of claims in Algorithm \ref{main.alg}]
Step $1^+$ and Step $1^-$ are Theorem \ref{thm5}.
So we consider Step $2$.
\par

Let 
$\pi = L(\Delta_{\rho_1}[x_1,y_1], \dots, \Delta_{\rho_r}[x_r,y_r]; \pi(\phi,\ep))$
and 
\begin{align*}
k_{\rho,z} 
&\coloneqq 
\#\{i \in \{1,\dots,r\} \;|\; \text{$\rho_i \cong \rho$, and, $x_i = z$ or $y_i  =-z$}\}
+m_\phi(\rho \boxtimes S_{2z+1})
\end{align*}
be as in Step $2$. 
By the assumption, 
we have $\min_{1 \leq i \leq r}\{x_i\} \in \{0,1/2\}$ so that $x_i \geq 0$ for $1 \leq i \leq r$.
Since $x_i \geq y_i$ and $x_i+y_i < 0$, we have $y_i < 0$ for $1 \leq i \leq r$.
Therefore, for $\rho \in \Cusp^\bot(\GL_d(F))$ and $x \in (1/2)\Z$ with $x \geq 0$, 
the multiplicity $M_{\rho,x}$ of $\rho|\cdot|^x$ in the extended cuspidal support $\esupp(\pi)$ is given by 
\[
M_{\rho,x} = \sum_{\substack{z \in (1/2)\Z \\ z \geq x}} k_{\rho,z}.
\]
In particular, 
\[
k_{\rho,x} = M_{\rho,x}-M_{\rho,x+1}.
\]
\par

Now suppose that $\pi$ is of Arthur type. 
Then by Theorem \ref{thm5} 
(or, more precisely, by Theorem 5.1 together with Algorithms 5.5 and 5.6 in \cite{At})
and the assumption in Step $2$, 
we obtain an extended multi-segment $\EE^*$ with $\pi \cong \pi(\EE^*)$ such that 
\[
([A,B]_\rho,l,\eta) \in \EE^* \implies B \in \{0,1/2\}.
\]
The associated $A$-parameter $\psi = \psi_{\EE^*}$ satisfies that 
\[
\rho \boxtimes S_{a} \boxtimes S_{b} \subset \psi \implies a-b \in \{0,1\}.
\]
In particular, one can write
\[
\psi_\EE = \bigoplus_\rho \bigoplus_{z \in (1/2)\Z_{\geq 0}} 
\left(\rho \boxtimes S_{z+1+\delta_z} \boxtimes S_{z+1-\delta_z}\right)^{\oplus m_{\rho,z}}
\]
for some multiplicity $m_{\rho,z} \geq 0$, 
where $\delta_{z} \in \{0,1/2\}$ with $z+\delta_z \in \Z$.
Since 
\[
S_{z+1+\delta_z} \otimes S_{z+1-\delta_z} 
\cong 
\left\{
\begin{aligned}
& S_1+S_3+\dots+S_{2z+1} \iif z \in \Z, \\
& S_2+S_4+\dots+S_{2z+1} \iif z \not\in \Z, 
\end{aligned}
\right. 
\]
by using Proposition \ref{ex.supp},
for $\rho \in \Cusp^\bot(\GL_d(F))$ and $x \in (1/2)\Z$ with $x \geq 0$, 
the multiplicity of $\rho|\cdot|^x$ in the extended cuspidal support $\esupp(\pi)$ is given by
\[
M_{\rho,x} = \sum_{\substack{z \in (1/2)\Z \\ z \geq x}} (z-x+1)m_{\rho,z}.
\]
Since
\[
k_{\rho,x} = M_{\rho,x}-M_{\rho,x+1} = \sum_{\substack{z \in (1/2)\Z \\ z \geq x}} m_{\rho,z}, 
\]
we have 
\[
m_{\rho,x} = k_{\rho,x} - k_{\rho,x+1}.
\]
In conclusion, if $\pi$ is of Arthur type, 
we must have $k_{\rho,x} \geq k_{\rho,x+1}$
for $\rho \in \Cusp^\bot(\GL_d(F))$ and $x \in (1/2)\Z$ with $x \geq 0$. 
In this case, the $A$-parameter $\psi$ is 
nothing but the one in Step $2$.
This completes the proof.
\end{proof}

\subsection{Proof of Theorem \ref{thm:s-eq}}\label{proof.eq}
Next we prove Theorem \ref{thm:s-eq}. 
The ``if'' part follows from \cite[Theorem 1.3]{X3}, \cite[Theorem 5.2]{At} 
and the definition of $\pi(\EE)$ (see \cite[Section 3.2]{At}). 
\par

We prove the ``only if'' part
by induction on the rank $n$ of $G_n$.
Let $\EE_1$ and $\EE_2$ be extended multi-segments for $G_n$ 
such that $\pi(\EE_1) \cong \pi(\EE_2) \not= 0$. 

\begin{description}
\item[Case 1]
Consider the case where $\EE_1$ and $\EE_2$ are \emph{non-negative}, i.e., 
for $i \in \{1,2\}$, 
\[
([A,B]_\rho,l,\eta) \in \EE_i \implies B \geq 0. 
\]
Then we show that 
$\EE_2$ can be obtained from $\EE_1$ by a finite chain of the operations (C), (UI) and their inverses. 
\par

If $D_{\rho|\cdot|^x}(\pi(\EE_1)) = 0$ for all $\rho \in \Cusp^\bot(\GL_d(F))$ and $x > 0$, 
then $D_{\rho|\cdot|^x}(\pi(\EE_1)) = 0$ for all $\rho \in \Cusp^\bot(\GL_d(F))$ and $x \in \R$ with $x \not= 0$.
In this case, by replacing $\EE_1$ and $\EE_2$ using Theorem \ref{thm5}, 
we may assume that 
for $i \in \{1,2\}$, 
\[
([A,B]_\rho,l,\eta) \in \EE_i \implies B = 0. 
\]
Then by considering the extended cuspidal support of $\pi(\EE_1)$, 
we can see that $\psi_{\EE_1} \cong \psi_{\EE_2}$. 
See the argument in Section \ref{proof.alg}. 
Hence we have $\EE_1 = \EE_2$. 
\par

Suppose that $D_{\rho|\cdot|^x}(\pi(\EE_1)) \not= 0$ 
for some $\rho \in \Cusp^\bot(\GL_d(F))$ and $x > 0$. 
Write
\[
\DD^+_\rho(\pi) = \left[(x,k_0),(x+1,k_1),\dots,(x+t-1,k_{t-1}); \pi^+\right]. 
\]
Then $t > 0$ and $x > 0$.
By replacing $\EE_1$ and $\EE_2$ using Theorem \ref{thm5}, 
we may assume that for $i \in \{1,2\}$, 
\begin{itemize}
\item
$\max\{ B \;|\; ([A,B]_\rho,l,\eta) \in \EE_i\} = x$; and
\item
the multi-set $\{ [A,B]_\rho \;|\; ([A,B]_\rho, l,\eta) \in \EE_i, \; B = x\}$ 
is exactly equal to 
\[
\bigsqcup_{j=1}^t \{ \underbrace{[x+j-1,x]_\rho, \dots, [x+j-1,x]_\rho}_{k_{j-1}-k_j} \}
\]
with $k_t \coloneqq 0$.
\end{itemize}
Let $\EE_i'$ be defined from $\EE_i$ by removing 
all extended segments of the form $([x+j-1,x]_\rho,l_j,\eta_j)$ for $1 \leq j \leq t$. 
Then by definition of $\pi(\EE)$, 
we see that $\pi(\EE_1) \cong \pi(\EE_2) \not= 0$ implies that $\pi(\EE'_1) \cong \pi(\EE'_2) \not= 0$. 
By the inductive hypothesis, 
$\EE_2'$ can be given from $\EE_1'$ by an operator 
which is a finite chain of the operations (C), (UI) and their inverses. 
Then $\EE_2$ can be given from $\EE_1$ by the same operator. 

\item[Case 2]
We prove the general case. 
For $i \in \{1,2\}$ and $z \in \Z$ with $z > 0$, 
let $\EE_{i,z}$ be given from $\EE_i$ by replacing 
each element $([A,B]_\rho,l,\eta) \in \EE_i$ with $([A+z,B+z]_\rho,l,\eta)$. 
Take $z \gg 0$ so that $\EE_{1,z}$ and $\EE_{2,z}$ are both non-negative.
Note that $\pi(\EE_{i,z}) \not= 0$ for $i \in \{1,2\}$.
Write
\begin{align*}
\pi(\EE_{1,z}) &= L(\Delta_{\rho_1}[x_1,y_1], \dots, \Delta_{\rho_r}[x_r,y_r]; \pi(\phi_z, \ep_z)), \\
\pi(\EE_{2,z}) &= L(\Delta_{\rho'_1}[x'_1,y'_1], \dots, \Delta_{\rho'_{r'}}[x'_{r'},y'_{r'}]; \pi(\phi'_z, \ep'_z)).
\end{align*}
Then by \cite[Theorem 3.6]{At}, 
\begin{align*}
\pi(\EE_{1}) &= L(\Delta_{\rho_1}[x_1-z,y_1+z], \dots, \Delta_{\rho_r}[x_r-z,y_r+z]; \pi(\phi, \ep)), \\
\pi(\EE_{2}) &= L(\Delta_{\rho'_1}[x'_1-z,y'_1+z], \dots, \Delta_{\rho'_{r'}}[x'_{r'}-z,y'_{r'}+z]; \pi(\phi', \ep')) ,
\end{align*}
where 
\begin{itemize}
\item
$\phi$ is given from $\phi_z$ 
by replacing each $\rho \boxtimes S_a \subset \phi_z$ with $\rho \boxtimes S_{a-2z}$; 
\item
$\ep(\rho \boxtimes S_{a-2z}) = \ep_z(\rho \boxtimes S_{a})$; and 
\item
$(\phi',\ep')$ is given from $(\phi'_z,\ep'_z)$ by the same way. 
\end{itemize}
Since $\pi(\EE_1) \cong \pi(\EE_2)$, we see that 
\begin{itemize}
\item
if we consider the multi-sets
$X = \{ \Delta_{\rho_1}[x_1-z,y_1+z], \dots, \Delta_{\rho_r}[x_r-z,y_r+z] \}$
and 
$X' = \{ \Delta_{\rho'_1}[x'_1-z,y'_1+z], \dots, \Delta_{\rho'_{r'}}[x'_{r'}-z,y'_{r'}+z] \}$, 
any element in $X \setminus (X \cap X')$ or $X' \setminus (X \cap X')$
is of the form $\Delta_\rho[-l+z, -l+1-z]$ or $\Delta_\rho[-l-1/2+z, -l+1/2-z]$
for some $l \in \Z$ with $l > 0$;
\item
as a virtual representation, 
$\phi_z - \phi_z'$ is a sum of representations of the form $\rho \boxtimes S_{2z}$
(with possibly negative multiplicities). 
\end{itemize}
Moreover, by definition and \cite[Theorem 1.3]{At}, 
for $i \in \{1,2\}$, 
if $([A,B]_\rho,l,\eta) \in \EE_i$, 
then 
\begin{itemize}
\item
$A+B \geq 0$; 
\item
$B+l \geq -1/2$; 
\item
if $B+l = -1/2$, then $\eta \in \{\pm1\}$ is uniquely determined. 
\end{itemize}
Therefore, for $i \in \{1,2\}$, 
one can define $\EE'_{i,z}$ from $\EE_{i,z}$
by repeatedly adding $([l-1+z,-l+z]_\rho,l,+1)$ for suitable $l > 0$
or $([l-1/2+z,-l-1/2+z]_\rho,l,+1)$ for suitable $l \geq 0$
such that $\pi(\EE'_{1,z}) \cong \pi(\EE'_{2,z})$. 
Let $\EE'_{i}$ be given from $\EE'_{i,z}$ by replacing 
each element $([A,B]_\rho,l,\eta) \in \EE'_{i,z}$ with $([A-z,B-z]_\rho,l,\eta)$.  
Then by construction, 
$\EE'_{i}$ is given from $\EE_i$ by applying the operation (P) several times. 
Since $\pi(\EE'_{1,z}) \cong \pi(\EE'_{2,z})$, by Case 1, 
$\EE'_{2,z}$ is given from $\EE'_{1,z}$ 
by a finite chain of the operations (C), (UI) and their inverses. 
Then $\EE'_2$ can be given from $\EE'_1$ by the same operator. 
In conclusion, 
$\EE_2$ can be given from $\EE_1$ 
by a finite chain of the operations (C), (UI), (P) and their inverses. 
\end{description}
This completes the proof of Theorem \ref{thm:s-eq}.


\end{document}